\documentclass[12pt, letterpaper]{amsart}

\usepackage{amssymb}
\usepackage{mathrsfs}

	\normalbaselines						  %

\newcommand{\bigslant}[2]{{\raisebox{.2em}{$#1$}\left/\raisebox{-.2em}{$#2$}\right.}}

\newcommand{\kf}[2]{{\begin{pmatrix} #1 \\ #2 \end{pmatrix}}}

\newcommand{\D}{{\mathbb  D}}

\newcommand{\T}{\mathbb{T}}

\newcommand{\Z}{{\mathbb  Z}}
\newcommand{\N}{{\mathbb  N}}
\newcommand{\C}{{\mathbb  C}}

\newcommand{\OZ}{{\mathbf{O}}}

\newcommand{\ID}{{\mathbf{1}}}

\newcommand{\OID}{{\mathbf{I}}}

\newcommand{\Hw}{\mathcal{G}}
\newcommand{\HI}{\mathcal{G}_\ast }
\newcommand{\HHw}{H^2(\mathcal X)}
\newcommand{\HHI}{\overline{H^2(\mathcal X)} }
\newcommand{\fdot}{\,\cdot\,}

\newcommand{\calC}{{\mathcal C}}
\newcommand{\cH}{\mathcal{H}}
\newcommand{\cD}{\mathcal{D}}
\newcommand{\cX}{\mathcal{X}}

\newcommand{\te}{\theta}

\newcommand{\Kft}{{\begin{pmatrix}H^2(\mathcal{E}_\ast)\\\bar\bigtriangleup L^2(\mathcal{E})\end{pmatrix}
\ominus \begin{pmatrix}\Theta\\\bigtriangleup\end{pmatrix} H^2(\mathcal{E})}}

\DeclareMathOperator{\clos}{clos}

\DeclareMathOperator{\spa}{span}

\DeclareMathOperator{\supp}{supp}

\newcommand{\ci}[1]{_{ {}_{\scriptstyle #1}}}
\newcommand{\ti}[1]{_{\scriptstyle \text{\rm #1}}}

\catcode`@=11
\chardef\mathlig@atcode\count255

\def\actively#1#2{\begingroup\uccode`\~=`#2\relax\uppercase{\endgroup#1~}}
\def\mathlig@gobble{\afterassignment\mathlig@next@cmd\let\mathlig@next= }
\def\mathlig@delim{\mathlig@delim}
\def\mathlig@defcs#1{\expandafter\def\csname#1\endcsname}
\def\mathlig@let@cs#1#2{\expandafter\let\expandafter#1\csname#2\endcsname}
\def\mathlig@appendcs#1#2{\expandafter\edef\csname#1\endcsname{\csname#1\endcsname#2}}

\def\mathlig#1#2{\mathlig@checklig#1\mathlig@end\mathlig@defcs{mathlig@back@#1}{#2}\ignorespaces}

\def\mathlig@checklig#1#2\mathlig@end{%
 \expandafter\ifx\csname mathlig@forw@#1\endcsname\relax
 \expandafter\mathchardef\csname mathlig@back@#1\endcsname=\mathcode`#1%
 \mathcode`#1"8000\actively\def#1{\csname mathlig@look@#1\endcsname}%
 \mathlig@dolig#1\mathlig@delim
\fi
\mathlig@checksuffix#1#2\mathlig@end
}

\def\mathlig@checksuffix#1#2\mathlig@end{%
\ifx\mathlig@delim#2\mathlig@delim\relax\else\mathlig@checksuffix@{#1}#2\mathlig@end\fi
}
\def\mathlig@checksuffix@#1#2#3\mathlig@end{%
\expandafter\ifx\csname mathlig@forw@#1#2\endcsname\relax\mathlig@dosuffix{#1}{#2}\fi
\mathlig@checksuffix{#1#2}#3\mathlig@end
}

\def\mathlig@dosuffix#1#2{%
\mathlig@appendcs{mathlig@toks@#1}{#2}%
\mathlig@dolig{#1}{#2}\mathlig@delim
}

\def\mathlig@dolig#1#2\mathlig@delim{%
 \mathlig@defcs{mathlig@look@#1#2}{%
 \mathlig@let@cs\mathlig@next{mathlig@forw@#1#2}\futurelet\mathlig@next@tok\mathlig@next}%
 \mathlig@defcs{mathlig@forw@#1#2}{%
  \mathlig@let@cs\mathlig@next{mathlig@back@#1#2}%
  \mathlig@let@cs\checker{mathlig@chck@#1#2}%
  \mathlig@let@cs\mathligtoks{mathlig@toks@#1#2}%
  \expandafter\ifx\expandafter\mathlig@delim\mathligtoks\mathlig@delim\relax\else
  \expandafter\checker\mathligtoks\mathlig@delim\fi
  \mathlig@next
 }%
 \mathlig@defcs{mathlig@toks@#1#2}{}%
 \mathlig@defcs{mathlig@chck@#1#2}##1##2\mathlig@delim{%
  \ifx\mathlig@next@tok##1%
   \mathlig@let@cs\mathlig@next@cmd{mathlig@look@#1#2##1}\let\mathlig@next\mathlig@gobble
  \fi
  \ifx\mathlig@delim##2\mathlig@delim\relax\else
   \csname mathlig@chck@#1#2\endcsname##2\mathlig@delim
  \fi
 }%
 \ifx\mathlig@delim#2\mathlig@delim\else
  \mathlig@defcs{mathlig@back@#1#2}{\csname mathlig@back@#1\endcsname #2}%
 \fi
}%

\catcode`@\mathlig@atcode

\mathchardef\ordinarycolon\mathcode`\:
\def\vcentcolon{\mathrel{\mathop\ordinarycolon}}
\mathlig{:=}{\vcentcolon=}
\mathlig{::=}{\vcentcolon\vcentcolon=}

\numberwithin{equation}{section}

\theoremstyle{plain}
\newtheorem{theo}{Theorem}[section]
\newtheorem{cor}[theo]{Corollary}

\theoremstyle{definition}

\theoremstyle{remark}
\newtheorem*{ex*}{Example}
\theoremstyle{remark}
\newtheorem*{exs*}{Examples}
\theoremstyle{remark}
\newtheorem*{rem}{Remark}
\theoremstyle{remark}
\newtheorem*{rems}{Remarks}

\title[Dilations and finite rank perturbations]{A geometric approach to finite rank unitary perturbations}
\author{Ronald~G.~Douglas}
\address{Department of Mathematics, Texas A\&M University, Mailstop 3368,      
 College Station, TX  77843, USA }
\email{rdouglas@math.tamu.edu}
\urladdr{http://www.math.tamu.edu/$\sim$ron.douglas}
\author{Constanze~Liaw}
\thanks{The second author is partially supported by the NSF grant DMS-1101477.}
\address{Department of Mathematics, Texas A\&M University, Mailstop 3368,      
 College Station, TX  77843, USA}
\email{conni@math.tamu.edu}
\urladdr{http://www.math.tamu.edu/$\sim$conni}

\keywords{Finite rank perturbations, rank one perturbations, dilation theory, normalized Cauchy transform}
 \subjclass[2010]{44A15, 47A20, 47A55}

\begin{document}

\begin{abstract}
For fixed $n\in\N$, we consider a family of rank $n$ unitary perturbations of a completely non-unitary contraction (cnu) with deficiency indices $(n,n)$ on a separable Hilbert space.

We relate the unitary dilation of such a contraction to its rank $n$ unitary perturbations. Based on this construction, we prove that the spectra of the perturbed operators are purely singular if and only if the operator-valued characteristic function corresponding to the unperturbed operator is inner. In the case where $n=1$ the latter statement reduces to a well-known result in the theory of rank one perturbations. However, our method of proof via the theory of dilations extends to the case of arbitrary $n\in\N$.

We find a formula for the operator-valued characteristic functions corresponding to a family of related cnu contractions.
In the case where $n=1$, the characteristic function of the original contraction we obtain a simple expression involving the normalized Cauchy transform of a certain measure. An application of this representation then enables us to control the jump behavior of this normalized Cauchy transform ``across'' the unit circle.
\end{abstract}

\maketitle

\section{Introduction}
Self-adjoint rank one perturbations occur naturally in mathematical physics \cite{weyl}. For example, a change in the boundary condition at the origin of a limit-point half-line Schr\"odinger operator from Dirichlet to Neumann, or to mixed conditions, can be reformulated as adding a rank one perturbation (see for example \cite{SIMREV}). Further, rank one perturbations have applications to the famous problem of Anderson localization \cite{SIMWOL}. And many fruitful connections with holomorphic composition operators, rigid functions and the Nehari interpolation problem have been discovered and exploited, see for example, \cite{poltsara2006}.

Although the perturbation is extremely restricted, rank one perturbations gave rise to a surprisingly rich theory, with the major difficulty being the instability of the singular part. In particular, the behavior of the so-called embedded singular continuous spectrum has proved notoriously hard to capture, see for example, \cite{Mixed, mypaper, SIMREV}.

Rank one unitary perturbations were generalized to higher rank unitary perturbations (for example, in \cite{BallLubin, Kapustin-Polt}), which describe the situation when the perturbed operator is not cyclic.

The problem at hand can be investigated in three settings: (a) In formulating the problem on an abstract Hilbert space, we can apply classical perturbation theory. (b) The subtleties of this very restrictive perturbation problem can be captured conveniently in its spectral representation, making the spectral measure the main object of interest. (c) The third perspective to finite rank unitary perturbations involves the model and dilation theory of Sz.-Nagy and Foia\c s, providing a rather geometric point of view. While we use all three approaches, the main result of this paper rests on (c) in order to obtain statements in the language of (b).

If the perturbed unitary operators have purely singular spectrum, then the situation is much simpler. Hence, it is useful to know exactly when this simpler case occurs.
For rank one perturbations, the following four statements are equivalent (see Theorems \ref{NTFAE} and \ref{t-delta} below for a more rigorous formulation):
\begin{itemize}
\item[(1)] The rank one unitary perturbations each have a purely singular spectrum.
\item[(2)] The corresponding cnu contraction operator $T$ is of class $C_{\fdot 0}$, or $(T^*)^n$ converges to the zero operator in the strong operator topology as $n\to \infty$.
\item[(3)] The corresponding functional model reduces to the ``one-valued model space'' instead of consisting of pairs of functions.
\item[(4)] The corresponding characteristic function for $T$ is inner; that is, the absolute value of its non-tangential boundary values on the unit circle equals $1$ a.e.~Lebesgue measure $m$.
\end{itemize}

In the higher rank setting, the analogue of (2) $\Leftrightarrow$ (4) was established by Ball--Lubin in \cite{BallLubin}, and the analogue of (3) $\Leftrightarrow$ (4) is trivial by virtue of the definition of the model space, see equation \eqref{e-KTe} below.

The main result of this paper, Theorem \ref{t-main}, provides the higher rank analogue of (1) $\Leftrightarrow$ (4), see Corollary \ref{c-main}.
The proof rests on relating the unitary dilation of the cnu contraction to its unitary perturbation, and then applying the result from dilation theory which states that the minimal unitary dilation of a cnu contraction is purely absolutely continuous.
A generalization of the ``classical'' proof of this equivalence from $n=1$ to $n>1$ does not seem to be straightforward.
However, it would still be desirable to extend the original proof to higher rank perturbations, since such an extension would likely yield the explicit density functions for the absolutely continuous part.

Let us briefly outline the larger framework of this paper while leaving the details and precise definitions to the preliminaries below.
Consider a cnu contraction on a separable Hilbert space with deficiency indices $(n,n)$. Via model theory, such an operator has a characteristic operator-valued function $\Theta$. Further, the operator is unitarily equivalent to the compression to the Sz.-Nagy--Foia\c s model space, $\mathcal{K}_\Theta$, of the operator obtained as the orthogonal direct sum of the shift operator on a vector-valued Hardy space with an absolutely continuous unitary operator, defined using $\Theta$, to be multiplication by the independent variable on the unit circle (see Theorem \ref{t-tte} below).

If we replace the action of the cnu contraction between its defect spaces by a unitary operator between the defect spaces, we will have constructed a rank $n$ unitary perturbation of the original cnu contraction. Since the defect spaces are $n-$dimensional, the choice of such a rank $n$ unitary perturbation is parametrized by the set of unitary $n\times n-$matrices, once each of the defect spaces has been identified with $\C^n$.

Section \ref{s-prelims} contains the preliminaries. In Subsection \ref{ss-GenMod}, we introduce the precise setup for rank $n$ unitary perturbations of a given cnu contraction which are the subject of study in Sections \ref{s-main} and \ref{s-reprtheta}.
In Section \ref{s-main} we take the minimal unitary dilation of the cnu contraction and explicitly express the unitary dilation using suitable orthogonal decompositions of its initial and range spaces. These particular decompositions allow us to construct the rank $n$ unitary perturbations of the cnu contraction by perturbing the unitary dilation with a rank $2n$ operator and then restricting the new operator to the smaller, original Hilbert space (on which the cnu contraction is defined). We then apply results from dilation theory.
In Section \ref{s-reprtheta}, a formula for the operator-valued characteristic functions of each operator in a related family of cnu contractions are derived by means of its representation from model theory.
In Section \ref{s-beta}, we focus on rank one unitary perturbations. The formula for the characteristic operator-valued function reduces to a characteristic function for each operator in the family  of rank one perturbations in which each is a cnu contraction. A simple formula for a representation of the characteristic function of the original cnu contraction is then derived. This formula involves the normalized Cauchy transform of a certain measure. Many corollaries are obtained. In particular, in Section \ref{s-jump} we capture the jump behavior of the latter normalized Cauchy transform ``across'' the points in the absolutely continuous spectrum of the perturbed unitary operator.

\section{Preliminaries}\label{s-prelims}

\subsection{Model theory and rank $n$ unitary perturbations}\label{ss-GenMod}
Consider Hilbert spaces $\mathcal{E}$ and $\mathcal{E}_\ast$ with $\dim(\mathcal{E}) = \dim(\mathcal{E}_\ast)=n<\infty$. Let $\Theta\in H^\infty(\mathcal{E}\to \mathcal{E}_\ast)$ with operator norm $\|\Theta(z)\|\ci{\mathcal{E} \to \mathcal{E}_\ast}\le 1$ for all $z\in\D$. Without loss of generality, we assume $\Theta(0)$ is the zero operator. Then the non-tangential boundary limits of $\Theta$ exist in the operator norm topology a.e.~Lebesgue $m$. We denote this limit also by $\Theta$; that is,
\[
\Theta(\xi) = \lim_{z\to \xi, z\in \D}\Theta(z)
\qquad \text{for }\xi \in \T
\text{ a.e.~}m.
\]
An operator-valued function $\Theta$ is said to be \emph{inner}, if $\Theta(\xi)$ is a unitary operator for $\xi\in\T$ a.e.~$m$.

The Sz.-Nagy--Foia\c s model space is given by
\begin{align}\label{d-Ktea}
\mathcal{K}_\Theta=\Kft,
\end{align}
where
\begin{align}\label{d-Kteb}
\bigtriangleup(\xi) = (\OID_\mathcal{E}-\Theta^\ast(\xi)\Theta(\xi))^{1/2} \text{ and }  \bar\bigtriangleup(\xi) = \left\{
\begin{array}{ll}
0, & \text{if }\bigtriangleup(\xi) = \OZ, \\ 1, & \text{if }\bigtriangleup(\xi) \neq\OZ
\end{array}
\right\}, \xi\in\T.
\end{align}
At this, $\OID_\mathcal{E}$ denotes the identity operator on $\mathcal{E}$.

Notice that there is a unitary transformation between $\mathcal{K}_\Theta$ as defined above and the quotient space
\[
\bigslant{(H^2(\mathcal{E}_\ast) \oplus \bar\bigtriangleup L^2(\mathcal{E}))}{\{\Theta f\oplus\bigtriangleup f: f\in H^2(\mathcal{E})\}}
\]
which intertwines the corresponding operators.

Clearly, the model space reduces to $\mathcal{K}_\Theta = H^2(\mathcal{E}_\ast) \ominus \Theta H^2(\mathcal{E})$ if and only if $\Theta$ is inner or $\bigtriangleup \equiv 0$.

Let $P_\Theta$ denote the orthogonal projection from $\begin{pmatrix}H^2(\mathcal{E}_\ast)\\ \bar\bigtriangleup L^2(\mathcal{E})\end{pmatrix}$ onto $\mathcal{K}_\Theta$. Further, consider the operators $T_z$ and $M_\xi^{\bar\bigtriangleup}$ acting as multiplication by the independent variables $z\in\D$ and $\xi\in\T$, respectively.
A \emph{completely nonunitary} (cnu) contraction is an operator $X$ that satisfies $\|X\|\le 1$ and is not unitary on any of its invariant subspaces.
One can show that the following contraction is cnu:
\begin{align}\label{d-cnu}
T_\Theta = P_\Theta \begin{pmatrix} T_z & \OZ\\ \OZ & M_\xi^{\bar\bigtriangleup} \end{pmatrix}
\qquad\text{on }\mathcal{K}_\Theta.
\end{align}

The class $C_{\cdot 0}$ consists of contractions $X$ such that $(X^\ast)^n$ converges to the zero operator in the strong operator topology. It is straightforward to see that $T_\Theta \in C_{\cdot 0}$ if and only if $\Theta$ is inner.

Let us explain briefly how the above construction allows the study of all cnu contractions on a separable Hilbert space. Consider a contraction $U_\OZ $ on a separable Hilbert space $\cH$. The \emph{defect operators} of $U_\OZ $ and $U_\OZ ^\ast$ are
\begin{align*}
D = (\OID- U_\OZ ^\ast U_\OZ )^{1/2}
\qquad
\text{and}
\qquad
D_\ast = (\OID- U_\OZ  U_\OZ ^\ast)^{1/2},
\end{align*}
respectively. And the \emph{defect spaces} are
\[
\cD := \clos D \cH\qquad\text{and}
\qquad
\cD_\ast := \clos D_\ast \cH,
\]
respectively.
The operator $U_\OZ $ is said to have \emph{deficiency indies} $(n,m)$, if $n = \text{dim}(\cD)$
and $m = \text{dim}(\cD_\ast)$, where $0\le m,n \le \aleph_0$ (that is, $m$ and $n$ are non-negative and countable).

\begin{theo}\label{t-tte}(see for example, page 253 of \cite{SzNF2010})
For any given cnu contraction $U_\OZ $ with deficiency indices $(n,n)$ on a separable Hilbert space $\cH$, there exists an operator-valued function $\Theta\in H^\infty (\cD\to\cD_\ast)$ such that
\[
(U_\OZ  \text{ on }\cH) \quad \cong \quad (T_\Theta \text{ on } \mathcal{K}_\Theta);
\]
that is, the operators are unitarily equivalent.
\end{theo}

The function $\Theta$ is called the \emph{characteristic operator function} of the cnu contractions $U_\OZ $ and $T_\Theta$. One can show that the formula
\begin{align*}
\Theta(z) =  -U_\OZ  +z D_\ast (\OID - z U_\OZ ^*)^{-1} D
\end{align*}
holds, where $\Theta(z):\cD\to\cD_\ast$ and $\OID$ is the identity operator on $\cH$.

The characteristic operator function is unique up to equivalence. In other words, given the pairs $\Theta_1$, $T_{\Theta_1}$ and $\Theta_2$, $T_{\Theta_2}$, the operators $T_{\Theta_1}$ and $T_{\Theta_2}$ are unitarily equivalent (that is,
\[
(T_{\Theta_1} \text{ on } \mathcal{K}_{\Theta_1})\cong (T_{\Theta_2} \text{ on } \mathcal{K}_{\Theta_2}))
\]
if and only if there exist unitary operators $U: \cD_2\to\cD_1$ and $V:(\cD_2)_\ast\to(\cD_1)_\ast$ satisfying 
\[
\Theta_1(z) U = V \Theta_2(z)\qquad \text{for }z\in\D.
\]

Let us outline the relationship of rank one perturbations with model theory.  If $n=1$, we describe part of Aleksandrov--Clark theory, and also refer the reader to Subsection \ref{sss-ACTheory} below.
Let $U_\OZ$ be as above.
With the notation $\cD^\perp = \cH\ominus \cD$ and $\cD_\ast^\perp= \cH\ominus \cD_\ast$ and a rank $n$ unitary operator $A:\cD\to \cD_\ast$, we can construct a unitary operator $U_A$ on $\cH$ by setting
\[
\begin{array}{rcccl}
\quad U_A:&\cD^\perp&\longrightarrow&\cD_\ast^\perp&
\quad\qquad\text{with }U_A f=U_\OZ f\text{ for } f\in\cD^\perp\text{ and} \\
&\oplus&&\oplus\\
 U_A:&\cD&\longrightarrow&\cD_\ast &
\quad\qquad\text{by }U_A f=A f\text{ for } f\in\cD.
\end{array}
\]

Let us introduce rank $n$ unitary perturbations of $U_\OZ$ on $\cH$.
Define the operators $P$ and $P_\ast$ to be the orthogonal projections of $\cH$ onto $\cD$ and $\cD_\ast$,
respectively. Further, we use the notation $P^\perp = I-P$ and $P^\perp_\ast = I-P_\ast$.
Consider the family of perturbations of $U_\OZ$ given by
\begin{align}\label{d-UA}
 U_A = P^\perp_\ast U_\OZ P^\perp \oplus A P
 \qquad A: \cD\to \cD_\ast\text{ with }\|A\|\le 1.
\end{align}
It is not hard to see that $\cD_\ast$ is a cyclic subspace for $U_A$ for each contraction $\|A\|\le 1$; that is,
\[
\cH = \clos\spa \{U_A^{(l)} \cD_\ast: l\in\Z\}
\text{ where }
U_A^{(l)} 
= 
\left\{
\begin{array}{ll}
U_A^l &\text{for }l\ge 0,\\
(U_A^\ast)^{-l}&\text{for }l < 0.
\end{array}
\right.
\]
Of course, if $A$ is unitary, then $U_A^{(l)} = U_A^l$ for $l\in\Z$.

If $A$ is a strict contraction; that is, $\|A\|<1$, then the operator $U_A$ is a cnu contraction with the same deficiency indices and the same defect spaces as $U_\OZ$. If $A$ is the zero operator $\OZ: \cD\to \cD_\ast$, then we recover the original cnu contraction $U_\OZ$.

In the remainder of this subsection, we assume that $A$ is unitary; that is, $AA^\ast = \OID$.
Then the operator $U_A$ is a rank $n$ unitary perturbation of $U_\OZ$. In this case, we can find a nonnegative measure $\mu_A$ on the unit circle $\T$ which is the scalar-valued spectral measure for the operator $U_A$; that is, we have the direct  integral
\begin{align}\label{d-directsum}
\cH = \oplus \int_\T \cH(\xi) d\mu_A(\xi)
\qquad\text{with }
U_A=\oplus \int_\T \xi\, \OID_{\cH(\xi)} d\mu_A(\xi),
\end{align}
where $ \OID_{\cH(\xi)} $ denotes the identity operator on the Hilbert space $ \cH(\xi) $.
Since $\cD_\ast$ is a cyclic subspace for $U_A$ of minimal dimension, we have $\max \dim\cH(\xi) = n$.

By Theorem \ref{t-tte} we can re-write this perturbation as a problem on the model space
\begin{align}\label{e-KTe}
\mathcal{K}_\Theta = 
\begin{pmatrix}H^2(\cD_\ast)\vspace{.4mm}\\\bar\bigtriangleup L^2(\cD)\end{pmatrix}
\ominus \begin{pmatrix}\Theta\\\bigtriangleup\end{pmatrix} H^2(\cD),
\end{align}
where $\Theta$ is the characteristic operator function that corresponds to $U_\OZ$.

\begin{theo}\label{NTFAE}
The following three statements are equivalent:
\begin{itemize}
\item[1)] The operators $(U_\OZ^\ast)^n$ converge to the zero operator in the strong operator topology as $n\to \infty$.
\item[2)] The model space reduces to $\mathcal{K}_\Theta=H^2(\cD)\ominus\Theta H^2(\cD)$; that is, the second component collapses to $\{0\}$.
\item[3)] The characteristic operator function $\Theta$ is inner.
\end{itemize}
\end{theo}

\begin{rems}
a) The equivalency of statements 1) and 3) can be found in \cite{BallLubin}, see also Proposition 3.5 of \cite{SzNF2010}. The fact that 2) $\Leftrightarrow$ 3) is trivial by virtue of the definition \eqref{e-KTe} of $\mathcal{K}_\Theta$.\\
b) The equivalency of part 3) to the fact that ``the scalar-valued spectral measure of the unitary perturbation is purely singular'' is the main result of the work at hand. It is proven in Section \ref{s-main} below for general $n\in \N$. In the case $n=1$, the equivalence of the two statements is implied by a well-known result in Aleksandrov--Clark theory, see Theorem \ref{t-delta} below.
\end{rems}

\subsection{Dilation theory}\label{ss-Dil}
Consider a contraction $C$ on a separable Hilbert space $\cH$ with deficiency indices $(n,n)$. 
Recall that a \emph{(power) unitary dilation} of the operator $C$ on the Hilbert space $\cH$ is defined to be a unitary operator $U$ on a larger Hilbert space $\mathcal K$ such that for all $m\in\N$ one has
\[
P_\cH U^m|_\cH = C^m.
\]
A unitary dilation $U$ of $C$ is said to be \emph{minimal} if we have $\mathcal K = \clos\spa\{U^k \mathcal \cH:k\in \Z\}$.

\begin{theo}(see for example, \cite{SzNF2010}, page 154)
 For every contraction on a Hilbert space, there exists a unique minimal unitary dilation.
\end{theo}

The minimal unitary dilation of a fixed contraction is determined uniquely up to isomorphism.
Hence, without loss of generality, we will refer to it as \emph{the} minimal unitary dilation of $C$.

It is well-known (see for example, \cite{SzNF2010}) that the dilation space $\mathcal K$ decomposes into
$\mathcal K = \HI \oplus \cH \oplus \Hw,$ where
\begin{align*}
\Hw  = \sum_{n\ge 0}\oplus U^n \mathcal{X},
\,\,\,
\HI = \sum_{n< 0}\oplus  U^n \mathcal{X_\ast}
\qquad\text{and} \qquad
\mathcal{X} = \Hw\ominus U \Hw,
\,\,\,
\mathcal{X}_\ast = \HI\ominus U^\ast \HI.
\end{align*}
If $U$ is the minimal unitary dilation and $C$ has deficiency indices $(n,n)$, then we have $\dim(\mathcal X)=\dim(\mathcal X_\ast)=n$.

It was proven by Schreiber that the minimal unitary dilation
of a strict contraction (that is, $\|C\|<1$) has an absolutely continuous spectrum.
This hypothesis was later weakened by Sz.-Nagy--Foia\c s.

\begin{theo}\label{t-Schreiber}(see for example, \cite{SzNF2010}, page 154)
 If $C$ is a cnu contraction, then its minimal unitary dilation has an absolutely continuous spectrum.
\end{theo}

A formula for the spectral multiplicity of the minimal unitary dilation is known. Let $\Theta$ be the operator-valued characteristic function of the cnu contraction $C$. For the defect space $\cD\ci{C}$of $C$ and $\xi\in\T$, consider the space
\[
  \cX(\xi) = \{f \in \cD\ci{C} : \bigtriangleup(\xi) f \neq 0\}.
\]

\begin{theo}\label{t-spectralmult}(see for example, \cite{SzNF2010}, page 251)
 The spectral multiplicity of the minimal unitary dilation of the cnu contraction $C$ with defect indices
$\partial\ci C = \partial\ci{C^\ast} = n$ is equal to
\[
 n+\dim \cX(\xi),
\qquad \text{for }\xi\in\T\text{ a.e.~}m.
\]
\end{theo}

\subsection{Aleksandrov--Clark Theory}\label{sss-ACTheory}
Let us generalize Clark's setup to the case of non-inner characteristic functions.
This Aleksandrov--Clark setup, which we now describe, is unitarily equivalent to that in Subsection \ref{ss-GenMod} for $n=1$.
Take $\te\in H^\infty(\D)$ with $\|\te\|\ci{H^\infty}\le 1$ and $\te(0)=0$.

Recall the definition of the Sz.-Nagy-Foia\c s model space $\mathcal{K}_\te$ and of the cnu contraction $T_\te$.
All rank one unitary perturbations of $T_\te$ are given by
	\[
	\widetilde U_\gamma=T_\te + \gamma \left(\fdot, \begin{pmatrix}\bar z\te\\\bar z\bigtriangleup\end{pmatrix}\right)_{\mathcal{K}_\te}\begin{pmatrix}\ID_z\\0\end{pmatrix}
	\qquad\text{on }\mathcal{K}_\te\text{ for }|\gamma|=1,
	\]
where $ \left(\fdot, \begin{pmatrix}\bar z\te\\\bar z\bigtriangleup\end{pmatrix}\right)_{\mathcal{K}_\te}\begin{pmatrix}\ID_z\\0\end{pmatrix}$ denotes the rank one operator defined by
\[
\begin{pmatrix}f\\g\end{pmatrix} \mapsto \begin{pmatrix}\left[(f,\bar z \te)\ci{H^2} + (g,\bar z\bigtriangleup)\ci{L^2}\right]\ID_z\\0\end{pmatrix}.
\]
Note that $\Theta(0)=0$ implies that $\bar z\in H^2$. Moreover, there is an explicit identification of the defect spaces of $T_\te$ with the spaces $ \left<\begin{pmatrix}\bar z\te\\\bar z\bigtriangleup\end{pmatrix}\right>$ and $\left<\begin{pmatrix}\ID_z\\0\end{pmatrix}\right>$.

Let $\mu_\gamma$ denote the spectral measure of $\widetilde U_\gamma$ with respect to the cyclic vector $\kf{\ID_z}{0}$; that is,
\[
\left((\widetilde U_\gamma+z I)(\widetilde U_\gamma-z I)^{-1}\kf{\ID_z}{0}, \kf{\ID_z}{0}\right)\ci{\mathcal{K}_\te}= \int\ci\T \frac{\xi+z}{\xi-z} \,d\mu_\gamma(\xi)
\qquad\text{for }
z\in \C\backslash\T.
\]

In order to explain the equivalence of this setup to that discussed in Subsection \ref{ss-GenMod}, for the case $n=1$, let us recall that the Clark operator $\Phi_\gamma: \mathcal K_\te\to L^2(\mu_\gamma)$ for $\gamma\in \T$ intertwines $T_\theta$ and $U_\OZ$ and it maps the defect spaces $\cD\ci{T_\theta}, \cD\ci{T_\theta^\ast}\subset \mathcal K_\te$ of $T_\theta$ to those of $U_\OZ$; that is,
\[
\Phi_\gamma T_\theta = U_\OZ \Phi_\gamma
\qquad\text{and}\qquad
\Phi_\gamma \begin{pmatrix}\te\\ \bigtriangleup\end{pmatrix} = \bar\xi
\qquad\text{and}\qquad
\Phi_\gamma \begin{pmatrix}\ID_z\\0\end{pmatrix} = \ID_\xi.
\]

There is a one-to-one correspondence between the spectral families that arise from rank one unitary perturbations and the corresponding $\te\in H^\infty (\D)$ with $\|\te\|\ci{H^\infty} \le 1$. For $z\in \C\backslash\T$ we have by a Herglotz argument that
\[
 \frac{\gamma + \te(z)}{\gamma - \te(z)} = \int\ci\T \frac{\xi+z}{\xi-z}\,d\mu_\gamma(\xi),
\]
see for example, \cite{cimaross}.

For a complex-valued measure $\tau$, the \emph{Cauchy transform} $K$ of the measure $f\tau$ is given by
\begin{align}\label{d-Cauchy}
K(f\tau)(z)=\int\ci\T\frac{f(\xi)d\tau(\xi)}{1-\bar\xi z},\qquad
z\in\C\backslash \T, f\in L^1(|\tau|).
\end{align}

Combining the latter two equations, one can show that
\begin{align}\label{e-herglotz}
 K\mu_\gamma(z) = 
\frac{1}{1-\bar\gamma \te(z)}\,,\qquad
z\in \C\backslash \T.
\end{align}

The \emph{normalized Cauchy transform}
\begin{equation}\label{d-noCT}
\calC_{f\mu_\gamma}(z)= \frac{K(f\mu_\gamma)}{K\mu_\gamma}(z)\,,\qquad z\in\C\backslash \T, f\in L^1(\mu_\gamma),
\end{equation}
is an analytic function on $\C\backslash \T$. It is one of the central objects in the theory of rank one perturbations.

\begin{rem}
If $\te$ is inner, we have $\Phi_\gamma^\ast f= \calC_{f \mu_\gamma}$ for all $f\in L^2(\mu_\gamma)$. For inner $\theta$, the latter formula plus the next theorem by Poltoratski, is the key to many results in the subject.
\end{rem}

\begin{theo}\label{t-polt}(Poltoratski \cite{NONTAN}, also see \cite{NPPOL}) For a complex Borel measure $\tau$ on $\T$ and any $f\in L^1(|\tau|)$, the non-tangential limit of $\calC_{f\tau}$ exists with respect to the singular part a.e.~$|\tau|\ti{s}$ and
\[
\lim\limits_{\D\ni z\to \xi } \calC_{f\tau} (z) = f(\xi ) \qquad \text{for }\xi \in \T\text{ a.e.~}|\tau|\ti{s}.
\]
\end{theo}

Consider the Lebesgue decomposition $d\mu_\gamma = d(\mu_\gamma)\ti{ac} + d(\mu_\gamma)\ti{s}$, where $d(\mu_\gamma)\ti{ac} = h_\gamma dm$, $h_\gamma \in L^1(m)$.
Using analytic methods, one can prove the following formula.

\begin{theo}\label{t-delta}(see for example, \cite{cimaross}, Proposition 9.1.14)
The density of the absolutely continuous part of the spectral measure $\mu_\gamma$ is given by
\[
 h_\gamma(\xi) = \frac{1-|\theta(\xi)|^2}{|\gamma - \theta(\xi)|^2}\qquad \text{for }\xi\in\T\text{ a.e.~}m.
\]

In particular, the operators $\widetilde U_\gamma$, $|\gamma|=1$, all have purely singular spectrum if and only if $\te$ is inner.
\end{theo}

In \cite{AbaLiawPolt} the authors proved a certain Aronszajn--Krein type formula in the case of self-adjoint rank one perturbations which states that the normalized Cauchy transform of a certain measure (involving $\mu_\gamma$) is independent of $\gamma$. This result is easily transferred to the case of rank one unitary perturbations as follows.

We use the notation $\mu = \mu_1$. By virtue of the spectral theorem, there exists a unitary operator
\begin{align}\label{d-Vgamma}
V_\gamma:L^2(\mu)\to L^2(\mu_\gamma) \text{ such that }
V_\gamma U_\gamma =  M_\zeta V_\gamma\text{ and } V_\gamma \ID_\xi= \ID_\zeta,
\end{align}
 where the operator $M_\zeta$ acts as multiplication by the independent variable in $L^2(\mu_\gamma)$.
In \cite{mypaper} the analogue of $V_\gamma$ for self-adjoint rank one perturbations was investigated.
For a function $f\in L^2(\mu)$ consider
\[
 f_\gamma = V_\gamma f\in L^2(\mu_\gamma), \qquad |\gamma|=1.
\]

\begin{theo}\label{t-CTinv}(Aronszajn--Krein type formula, see \cite{AbaLiawPolt})
 For $f \in L^2(\mu)$ we have
 \[
 \frac{K(f\mu)}{K\mu} (z)= \frac{K(f_\gamma\mu_\gamma)}{K\mu_\gamma}(z)
 \qquad \text{for }z\in\C\backslash\T\text{ a.e.~}m.
 \]
\end{theo}

\section{Matrix representations of unitary dilations and rank $n$ unitary perturbations}\label{s-main}
Recall that $U_\OZ$ is a cnu contraction on $\cH$ with deficiency indices $(n,n)$, $n\in\N$, with characteristic operator function $\Theta$, that the operator-valued function $\bigtriangleup(\xi) = (\OID_\cD-\Theta^\ast(\xi)\Theta(\xi))^{1/2}$ is defined for $\xi\in\T$ a.e.~$m$.
Recall the definition of the rank $n$ perturbation $U_A$ of $U_\OZ$ given by \eqref{d-UA}. In this section, we assume that $A$ is a unitary operator, so $U_A$ is a unitary operator.
Finally, recall the definition of the space $\cX(\xi) = \{f \in \cD : \bigtriangleup(\xi) f \neq 0\}$ for $\xi \in\T$ a.e.~$m$, and let $N(\xi)$ denote the function of spectral multiplicity of the absolutely continuous part $(U_A)\ti{ac}$ of the perturbed operator.

\begin{theo}\label{t-main}
Consider the rank $n$ perturbation $U_A$ given by \eqref{d-UA} and assume that $A$ is unitary.
Then the spectral multiplicity function $N(\xi) = \dim \cX(\xi)$ for $\xi\in\T$ a.e.~$m$. In particular, the operator $U_A$ has no absolutely continuous part on a Borel set $B\subset\T$ if and only if $\dim \cX (\xi) = 0$ for $\xi\in B$ a.e.~$m$; or, equivalently, $\Theta(\xi)$ is unitary for $\xi \in B$ a.e.~$m$.
\end{theo}

Taking $B=\T$, we immediately obtain the main result of this paper. Recall that $\mu_A$ denotes the scalar-valued spectral measure associated with $U_A$ via the direct integral \eqref{d-directsum}.

\begin{cor}\label{c-main}
 Assume the above setting. The scalar-valued spectral measure $\mu_A$ is singular with respect to Lebesgue measure $m$ on the unit circle $\T$ if and only if $\dim \cX (\xi) = 0$ for $\xi\in \T$ a.e.~$m$. The latter is equivalent to the statement that $\Theta$ is inner.
\end{cor}

\begin{rem}
For $n=1$, this result is a weaker version of Theorem \ref{t-delta}.
Our proof is purely model theoretic and hence geometric, as opposed to the analytic proof of Theorem \ref{t-delta}.
\end{rem}

\begin{proof}[Proof of Theorem \ref{t-main}]
Assume the setting of Theorem \ref{t-main}. 
Dilation theory tells us that the minimal unitary dilation $W$ lives on a larger Hilbert space $\mathcal{K} \supset \cH$.
Recall that $\mathcal K = \HI \oplus \cH \oplus \Hw$, where $\Hw= \sum_{n\ge 0}\oplus W^n \mathcal{X}$ and $\mathcal{X} = \Hw\ominus W \Hw$, and analogously for $\HI$ and $\mathcal X_\ast$.

Let $w$ be the independent variable in $L^2(\mathcal X)$ and let $M_w$ denote the multiplication operator by the variable $w$. Without loss of generality, we have $\mathcal K = \HHI \oplus \cH \oplus \HHw$.
Indeed, we have the following unitary equivalences:
\[
(W  \text{ on }\Hw) \ \cong \ (M_w \text{ on } \HHw),
\qquad\text{and}\qquad
(W  \text{ on }\HI) \ \cong \ (M_w \text{ on } \HHI).
\]
In the second unitary equivalence, we have identified $\cX$ with $\cX_\ast$ which is possible, because $\dim \cX= \dim \cX_\ast = n<\infty$.
The reason for identifying the spaces $\cX$ and $\cX_\ast$ is that for the rank $2n$ perturbation, $\widetilde W$ (defined below in equation \eqref{d-tildeW}), of the action of $W$ on $\mathcal G\oplus \mathcal G_\ast$ we have 
\[
(\widetilde W  \text{ on }\Hw \oplus \HI) \ \cong \ (M_w \text{ on } L^2(\cX)).
\]

In what follows, we will further decompose the spaces $\HHI$, $\cH$ and $\HHw$; see \eqref{decomp1} below. In order to recall for the reader the precise structure of the dilation space $\mathcal{K}$, we add to the classical notation and use the symbol $\boxplus = \oplus$ for the orthogonal direct sum in the above decomposition; that is,
\[
\mathcal K = \HHI \boxplus \cH \boxplus \HHw.
\]
The space $\mathcal X$ is $n-$dimensional, because the unitary dilation $W$ is minimal.

In the initial and range spaces of $W$, we use the decompositions:
\begin{align}
\begin{array}{lrclclcl}\label{decomp1}
\text{initial space:}& \mathcal{K} 
&=&  M_{\bar w} \HHI   \oplus M_{\bar w}\mathcal X 
&\boxplus& \cD^\perp \oplus \cD 
&\boxplus& \HHw,
\vspace{.1cm}\\
\text{range space:}&  \mathcal{K}
 &=& \HHI 
 &\boxplus& \cD_\ast^\perp \oplus \cD_\ast 
 &\boxplus& \mathcal X \oplus M_w\HHw .
\end{array}
\end{align}

Then we have the following actions
\[
W|\ci{M_{\bar w} \HHI\boxplus\HHw} = M_{w}\qquad\text{and}\qquad
W|\ci{\cD^\perp} = U_\OZ|\ci{\cD^\perp},
\]
and any choice of rank $n$ unitary operators
\[
W|\ci{M_{\bar w}\mathcal X}:M_{\bar w}\mathcal X \to \cD_\ast
\qquad\text{and}\qquad
W|\ci{\cD}: \cD \to \mathcal X
\]
ensures that $W$ is the minimal unitary dilation of $U_\OZ$.

Recall that $A:\cD\to\cD_\ast$ is a rank $n$ unitary operator.
Consider the rank $2n$ unitary perturbation $\widetilde W$ of $W$ given by
\begin{align}\label{d-tildeW}
\widetilde W|\ci{M_{\bar w}\HHI\boxplus \cD^\perp\boxplus \HHw} = W|\ci{M_{\bar w}\HHI\boxplus \cD^\perp\boxplus \HHw},
\quad\widetilde W|\ci{M_{\bar w}\mathcal X } =  M_w
\quad\text{and}\quad
\widetilde W|\ci{\cD} = A.
\end{align}
Then the operator $\widetilde W$ acts as the bilateral shift when compressed to the reducing subspace
\[
L^2(\mathcal X)
 = M_{\bar w} \HHI   \oplus M_{\bar w}\mathcal X \boxplus\HHw 
= \HHI\boxplus \mathcal X \oplus M_w\HHw.
\]
Further, we have $\widetilde W|\ci{\cH} = U_A$. In other words, we have
\begin{align}
 \widetilde W \cong[\text{bilateral shift on }(\oplus L^2(\mathcal X))] \oplus [U_A\text{ on }\cH].
\end{align}

Let $n(\xi)$, $ n\ti{ac}(\xi)$ and $ \widetilde n\ti{ac}(\xi)$ denote the functions of spectral multiplicity of $W$, $W\ti{ac}$ and $\widetilde W\ti{ac}$, respectively, defined $\xi\in \T$ a.e.~$m$.

By virtue of Theorem \ref{t-Schreiber}, the minimal unitary dilation $W$ of the cnu contraction $U_\OZ$ has purely absolutely continuous spectrum; that is, $n(\xi) = n\ti{ac}(\xi)$, $\xi\in \T$ a.e.~$m$.

Further, we have $ n\ti{ac}(\xi) = \widetilde n\ti{ac}(\xi)$, $\xi\in \T$ a.e.~$m$, by the Kato--Rosenblum theorem (see for example, \cite{katobook}, or \cite{CP}). Indeed, the operator $\widetilde W$ is by definition a rank $2n$ perturbation of the minimal unitary dilation $W$.

Now, Theorem \ref{t-spectralmult} yields the following formula for the multiplicity of $\widetilde W\ti{ac}$:
\begin{align}\label{eq-1}
\widetilde n\ti{ac}(\xi) = n + \dim \cX (\xi)
\qquad \xi\in\T \text{ a.e.~}m.
\end{align}
In equation \eqref{eq-1}, the summand $n$ comes from the multiplicity of the bilateral shift on $L^2(\mathcal X)$.
In particular, the spectral multiplicity $N(\xi)$ of $(U_A)\ti{ac} = (\widetilde W)\ti{ac}|\ci{\cH}$ satisfies
\begin{align}\label{eq-2}
N(\xi) = \dim \cX (\xi)\text{ for }\xi\in\T \text{ a.e.~}m.
\end{align}

The second statement of the theorem follows immediately.
\end{proof}

\begin{rems}
(a) The construction of the unitary dilation $W$ and the perturbation $\widetilde W$ can be extended to the case where the defect operators of $U_\OZ$ are trace class. However, the argument from equation \eqref{eq-1} to \eqref{eq-2} does not hold true in that case.\\
(b) If we choose the actions of $W$ and $\widetilde W$ more carefully, it is possible to view $\widetilde W$ as a rank $n$ perturbation of $W$.
\end{rems}

\section{Operator-valued characteristic functions of $U_A$ for $\|A\|<1$}\label{s-reprtheta}
Let $U_A$ be defined by \eqref{d-UA}.
Assume that $A$ is a strict contraction; that is, $\|A\|<1$ which, since $n<\infty$, is equivalent to the statement that $\|Af\|=\|f\|$ for $f\in \cD$ implies $f=0$. As mentioned in Subsection \ref{ss-GenMod}, this implies that the perturbation $U_A$ is a cnu contraction with deficiency indices $(n,n)$. From model theory we know that $U_A$ has a corresponding operator-valued characteristic function $\Theta_A(z):\cD\to \cD_\ast$. The goal of this section is to find a representation of $\Theta_A(z)$.

For a fixed strict contraction $A:\cD\to \cD_\ast$, let $\{k_i\}_{i=1}^n$ and $\{\widetilde k_i\}_{i=1}^n$ be orthonormal bases of $\cD$ and $\cD_\ast$, respectively, for which the operator $A$ is ``diagonal''; that is, 
\[
Ak_i = \beta_i \tilde k_i
\qquad \text{for }\beta_i\in \D\text{ and } i=1,\hdots, n.
\]

Let $\OID$, $\OID_\cD$ and $\OID_{\cD_\ast}$ denote the identity operators on $\cH$, $\cD$ and $\cD_\ast$, respectively.
Notice that the operator-valued characteristic function is a holomorphic function from $\D$ to the rank $n$ linear operator. We use the standard notation $|A|^2 = A^\ast A$.

\begin{theo}\label{t-char}
Assume the above setting for $U_A$, $A$, $\{k_i\}_{i=1}^n$ and $\{\widetilde k_i\}_{i=1}^n$. The action of the operator-valued characteristic function $\Theta_A(z)$ is
\begin{align}\label{formula}
(\Theta_A k_j)(z)
=
-\beta_j k_j + z \oplus_{i=1}^n  \frac{ (1 - |\beta_i|^2)\langle(\OID - z U_\OZ^\ast )^{-1}k_i, \tilde k_i \rangle}{\langle(\OID - z U_\OZ^\ast )^{-1} \tilde k_i , \tilde k_i \rangle - z \bar\beta_i  \langle(\OID - z U_\OZ^\ast )^{-1}k_i , \tilde k_i \rangle}  \, \tilde k_i
\end{align}
for $j=1,\hdots,n$ and $z\in\D$.
\end{theo}

 \begin{rem}
A similar but more complicated formula was obtained in \cite{BallLubin, fuhrmann}. In particular, the operator function $S(z)$ that occurs on the right hand side of their formula is related to the operator-valued characteristic function $\Theta_A(z)$ via a unitary change of basis, see \cite{BallLubin}. The authors used their formula to prove some interesting spectral properties of $U_A$  and the higher rank unitary perturbations.
 \end{rem}

\begin{proof}
Take $z\in\D$ and recall that the operator-valued characteristic function of the cnu contraction $U_A$ is given by
 \[
 \Theta_A(z) = -U_A +z D\ci{U_A^*} (\OID - z U_A^*)^{-1} D\ci{U_A},
 \]
where $D\ci{U_A}= (\OID_{\cD} - U_A^\ast U_A)^{1/2}$ and $D\ci{U_A^\ast}= (\OID_{\cD_\ast} - U_A U_A^\ast)^{1/2}$ are the defect operators of $U_A$ and $U_A^\ast$, respectively.

From the definition \eqref{d-UA} of $U_A$ we obtain
\begin{align*}
 D\ci{U_A}  = (\OID_{\cD} - |A|^2)^{1/2} P
 \qquad\text{and}\qquad
D\ci{U_A^\ast} = (\OID_{\cD^\ast} - |A^\ast|^2)^{1/2} P_\ast,
\end{align*}
where $P$ and $P_\ast$ are the orthogonal projections of $\cH$ onto $\cD$ and $\cD_\ast$, respectively.
Fix $j=1, 2, \hdots, n$.
By the above choice of bases, $\{k_i\}$ and $\{\tilde k_i\}$, we have
 \begin{align}\label{key}
 (\Theta_A k_j)(z)
 = -\beta_j \tilde k_j + z \oplus (1 - |\beta_i|^2) (1 - |\beta_i|^2) \langle x_i,\tilde k_i \rangle  \tilde k_i, 
 \end{align}
 where $x_i = (\OID - z U_A^\ast)^{-1}k_i \in \cH$, $z\in\D$.

By the definition of $U_A$, we have $U_A^\ast = U_\OZ^\ast + (A^\ast - U_\OZ^\ast) P_\ast$. Hence, we obtain
\begin{align*}
 k_i 
= [(\OID-zU_\OZ^\ast ) x_i + z \langle x_i, \tilde k_i \rangle (U_\OZ^\ast -A^\ast) \tilde k_i ],
 \end{align*}
 and so
\[ 
x_i
 = 
 (\OID-zU_\OZ^\ast )^{-1}[k_i - z\langle x_i, \tilde k_i \rangle (U_\OZ^\ast -A^\ast) \tilde k_i ].
\]
 
 Taking the inner product of both sides with $\tilde k_i $ we have
 \begin{align*}
 \langle x_i,\tilde k_i \rangle = \langle(\OID-zU_\OZ^\ast )^{-1}k_i, \tilde k_i \rangle - z\langle x_i, \tilde k_i \rangle \langle(\OID-zU_\OZ^\ast )^{-1}(U_\OZ^\ast -A^\ast) \tilde k_i , \tilde k_i \rangle.
 \end{align*}
Solving for $\langle x_i,\tilde k_i \rangle$ yields
\[
 \langle x_i,\tilde k_i \rangle = 
\frac{\langle(\OID - z U_\OZ^\ast )^{-1}k_i, \tilde k_i \rangle}{1+ z \langle(\OID - z U_\OZ^\ast )^{-1}(U_\OZ^\ast  - A^\ast) \tilde k_i , \tilde k_i \rangle} \,.
\]
Recall that $A^\ast \tilde k_i = \bar\beta_i k_i$, and using the series expansion for $(\OID - z U_\OZ^\ast )^{-1}$, we obtain
\[
1+ z \langle(\OID - z U_\OZ^\ast )^{-1}U_\OZ^\ast\tilde k_i , \tilde k_i \rangle=\langle(\OID - z U_\OZ^\ast )^{-1} \tilde k_i , \tilde k_i \rangle.
\]
Therefore, we have
\[
 \langle x_i,\tilde k_i \rangle = 
\frac{\langle(\OID - z U_\OZ^\ast )^{-1}k_i, \tilde k_i \rangle}{\langle(\OID - z U_\OZ^\ast )^{-1} \tilde k_i , \tilde k_i \rangle - z \bar\beta_i  \langle(\OID - z U_\OZ^\ast )^{-1}k_i , \tilde k_i \rangle} .
\]

The statement in the theorem follows when we substitute the latter expression into equation \eqref{key}.
\end{proof}

\section{Characteristic functions in the case $n=1$}\label{s-beta}
In the remainder of this paper we consider the case of $n=1$.
In definition \eqref{d-UA} we parametrized the rank one contraction operators $A:\cD\to \cD_\ast$ by $\gamma\in\D\cup\T$. We use the notation $U_\gamma$ for the perturbed operator.

The following setup is unitarily equivalent to the setting for the Aleksandrov-Clark theory described in Subsection \ref{sss-ACTheory}. 
We state our results using the spectral representation of the operator $U_1$; that is, on the space $L^2(\mu)$, where $\mu=\mu_1$ denotes the spectral measure of the unitary operator $U_1$ with respect to the cyclic vector $\tilde k\in\cD_\ast$.
We denote the independent variable of $L^2(\mu)$ by $\xi$, the operator which acts as multiplication by the independent variable by $M_\xi$, the function in $L^2(\mu)$ defined so that $\bar\xi(\xi) = \bar\xi$ by $\bar\xi$, and the constant function identically equal to $1$ in $L^2(\mu)$ by $\ID_\xi$.
Then the cnu contraction $\hat U_0$ on $L^2(\mu)$ (which is unitarily equivalent to $U_\OZ$ in Subsection \ref{ss-GenMod} and $\widetilde U_0$ in Subsection \ref{sss-ACTheory}) is given by
\begin{align}\label{f-U0}
\hat U_0&= M_\xi - ( \fdot, \bar \xi )\ci{L^2(\mu)}\ID_\xi\quad\text{on }L^2(\mu).
\end{align}
The defect spaces of $\hat U_0$ are $\cD\ci{\hat U_0} = \spa\{\bar\xi\}$ and $\cD\ci{\hat U_0^*} = \spa\{\ID_\xi\}$. (Note we have implicitly identified them by sending $\bar\xi$ to $\ID_\xi$.)

The rank one unitary perturbations of $\hat U_0$ are given by
\begin{align}\label{f-Ugamma}
\hat U_\gamma = \hat U_0 + \gamma ( \fdot,\bar\xi)\ci{L^2(\mu)} \ID_\xi, \qquad |\gamma|=1.
\end{align}
Alternatively, we can write $\hat U_\gamma = M_\xi +(\gamma-1) ( \fdot,\bar\xi)\ci{L^2(\mu)} \ID_\xi$ as a rank one perturbation of the unitary operator $M_\xi$.

\subsection{Representation for the characteristic functions for $\hat U_\beta$, $|\beta|<1$}\label{ss-RepBeta}
The family of operators $U_\beta$ given by \eqref{f-Ugamma} consists of cnu contractions for all $|\beta|<1$.
(We use the symbol $\beta$ in this subsection to avoid possible confusion between
the cases where the parameter is in the unit disc, $|\beta|<1$, and where the parameter
is in the unit circle; that is, $|\gamma|=1$.)
Each such $U_\beta$ has a characteristic function which we denote by $\Theta_\beta$.

Let us identify the one dimensional defect spaces $<\bar\xi>$ and $<\ID_\xi>$ with the complex numbers; that is, we consider the characteristic function $\te_\beta\in H^\infty(\D)$ that satisfies $(\Theta_\beta \bar\xi)(z) = \theta_\beta(z)\ID_\xi$.

Recall that the Cauchy transform of a complex-valued measure $f\tau$ on the unit circle is given by
$K(f\tau)(z)=\int\ci\T\frac{f(\xi)d\tau(\xi)}{1-\bar\xi z}$ for $z\in\C\backslash \T$ and $f\in L^1(|\tau|)$.

\begin{theo}\label{t-char3}
The characteristic functions in the above setup can be expressed in terms of the Cauchy transform
of the measures $\mu$ and $\bar\xi\mu$; namely, we have
\begin{align*}
\theta_\beta (z) = -\beta + z (1-|\beta|^2)\frac{K(\bar\xi\mu)}{K\mu-z\bar\beta K(\bar\xi\mu)}\,,
\qquad |\beta|<1, z\in \D.
\end{align*}
\end{theo}

For the convenience of the reader who is not interested in the rank $n$ case, we provide two proofs. In the first proof, we argue that the representation \eqref{formula} of the operator-valued characteristic function reduces to the desired formula. The alternative proof does not depend on Sections \ref{s-main} and \ref{s-reprtheta} and is based on a direct computation, mimicking the steps in the proof of Theorem \ref{t-char}.

\begin{proof}[Proof of Theorem \ref{t-char3}]
Recall that $\cD\ci{U_0}$ is one-dimensional. Hence, any $k\in \cD$ can be represented by $k= c\bar\xi$ for some $c\in \C$.
We obtain $Ak= \beta\tilde k$, or $A:\cD\to\cD_\ast: \bar  {\xi} \mapsto \beta\ID_\xi$ by linearity. In the numerator of the formula in Theorem \ref{t-char}, one can see that $\langle(\OID - z U_\OZ^\ast)^{-1}k, \tilde k\rangle = \langle(1-z\bar  {\xi})^{-1}\bar  {\xi}, \ID_\xi\rangle = K(\bar  {\xi}\mu)$ and similarly in the denominator.
\end{proof}

\begin{proof}[Alternative proof of Theorem \ref{t-char3}]
 Recall that
$\hat U_\beta = M_\xi - (1-\beta) <\cdot, \bar\xi> \ID_\xi
= P\ci{<\ID_\xi>^\perp}\hat U_0 |\ci{<\bar\xi>^\perp} + \beta <\cdot, \bar\xi> \ID_\xi$.
With this identity, it is not hard to compute the defect operators $D\ci{\hat U_\beta^\ast} = (1-|\beta|^2)^{1/2} P\ci{<\ID_\xi>}$ and $D\ci{\hat U_\beta} = (1-|\beta|^2)^{1/2} P\ci{<\bar\xi>}$.
With the representation
$\Theta_\beta(z) = - \hat U_\beta +z D\ci{\hat U_\beta^\ast} (\OID - z \hat U_\beta^\ast)^{-1} D\ci{\hat U_\beta}$
of the characteristic operator function of $\hat U_\beta$, we obtain
\begin{align}
 (\Theta_\beta \bar\xi)(z) 
&=[ -\beta +z (1-|\beta|^2) <x, \ID_\xi>]\, \ID_\xi,
\label{e-thetabeta}
\end{align}
where we have used the notation
\begin{align}\label{d-x}
 x = (\OID - z \hat U_\beta^\ast)^{-1} \bar\xi.
\end{align}

From the latter identity and the definition of $\hat U_\beta$, it follows that
\[
 \bar\xi = (\OID - z M_{\bar\xi}) x + z(1-\bar\beta) <x,\ID_\xi> \bar\xi.
\]
Solving for $x$ (in terms of $<x,\ID_\xi>$ and the other variables), we obtain
\[
 x = \left[ 1 - z(1-\bar\beta) <x,\ID_\xi> \right] (\OID - z M_{\bar\xi})^{-1} \bar\xi.
\]
Taking the inner product of the latter identity with the function $\ID_\xi$ and solving for $<x,\ID_\xi>$ yields
\begin{align}\label{e-xid}
 <x,\ID_\xi> 
 = \frac{K\bar\xi\mu}{1+z(1-\bar\beta) K\bar\xi\mu}
 = \frac{K\bar\xi\mu}{K \mu+\bar\beta z K\bar\xi\mu}\,.
\end{align}

Combining \eqref{e-thetabeta}, \eqref{e-xid} and the fact that
$\theta_\beta(z)= \theta_\beta(z) \ID_\xi = (\Theta_\beta  \bar\xi)(z)$
we obtain the statement of Theorem \ref{t-char3}.
\end{proof}

Theorem \ref{t-char3} and Theorem \ref{t-char1} (below) plus some simple algebra yield
the following corollary.

\begin{cor}(Livsic \cite{Livsic})
The characteristic function $\theta_\beta$ is related to $\theta = \theta_0$ via the linear fractional transformation
\begin{align*}
\theta_\beta (z) = \frac{-\beta+\theta(z)}{1-\bar\beta \theta(z)}\,.
\end{align*}
In particular, the characteristic functions $\theta$ and $\theta_\beta$ are inner simultaneously. And if $\theta_{\beta_0}$ is a finite Blaschke product for some $\beta_0\in \D$, then all characteristic functions $\theta_\beta$ are finite Blaschke products.
\end{cor}

\begin{rem}
The relation in the latter corollary is essentially analogous to formula (22)
in Livsic \cite{Livsic} in the case where the deficiency indices are $(1,1)$ and his parameter $\tau = -\beta$. The main focus of Livsic's paper was to find necessary and sufficient conditions for two simple partial isometries to be unitarily equivalent.
While the paper includes several results about the spectrum of certain unitary extensions, he did not consider rank one perturbations.
\end{rem}

\subsection{Representation for the characteristic function $\theta$ of $\hat U_0$}\label{s-reprtheta1}
We find a very simple representation of the characteristic function of $\hat U_0$ involving the normalized Cauchy transform. We derive several results from this simplification.

\begin{theo}\label{t-char1}
The characteristic function $\theta(z)$ of the cnu contraction $U_0$ is given by
\begin{align}\label{e-vartheta}
\theta(z) = z \calC_{\bar\xi\mu}= \gamma z \calC_{\bar\xi\mu_\gamma}\qquad\text{for }  z\in\D, |\gamma|=1.
\end{align}
\end{theo}

\begin{proof}
The first equality follows immediately from Theorem \ref{t-char}.

For the second equality recall the definition of $V_\gamma:L^2(\mu)\to L^2(\mu_\gamma)$
and that $V_\gamma\ID_\xi=\ID_\zeta$ (see equation \eqref{d-Vgamma}). Here $\zeta$ denotes the independent variable of $L^2(\mu_\gamma)$.
Observe that
\begin{align}\label{e-trivial}
V_\gamma \bar\xi = \gamma\bar \zeta .
\end{align}
Indeed, the intertwining relationship $V_\gamma \hat U_\gamma = M_\zeta V_\gamma$ implies $V_\gamma \hat U_\gamma^\ast = M_{\bar \zeta} V_\gamma$. And together with the identity $\hat U_\gamma^\ast = M_{\bar\xi} + (\bar\gamma - 1) (\fdot, \ID_\xi)\ci{L^2(\mu)} \bar\xi$, it follows that $V_\gamma \bar\xi = V_\gamma M_{\bar\xi} \ID_\xi = M_{\bar \zeta} V_\gamma \ID_\xi + V_\gamma (1-\bar\gamma) (\ID_\xi,\ID_\xi)\ci{L^2(\mu)} \bar\xi  = \bar \zeta +  (1-\bar\gamma) V_\gamma \bar\xi$. It remains to solve for $V_\gamma\bar\xi$ and recall that $|\gamma|=1$.

For the expression in the middle of \eqref{e-vartheta} we use the fact that the Cauchy transform is invariant under a change of the perturbation parameter $\gamma$ by Theorem \ref{t-CTinv}. This, together with the identity \eqref{e-trivial}, yields $ \calC_{\bar\xi\mu} = \calC_{(V_\gamma \bar\xi)\mu_\gamma} =   \calC_{\gamma\bar\xi\mu_\gamma} = \gamma \calC_{\bar\xi\mu_\gamma}$, which completes the proof.
\end{proof}

\begin{rems}
(a) Equation \eqref{e-trivial} is a simple case of the representation theorem in \cite{mypaper}. We decided to provide a direct proof of this case for the convenience of the reader, since the representation theorem is formulated in the case of self-adjoint rank one perturbations there.\\
(b) In the case where the spectral measure $\mu$ is purely singular, the two representations of the characteristic function (the first equality of \eqref{e-herglotz} and \eqref{e-vartheta} for $\gamma =1$) agree verbatim.
\end{rems}

The following three corollaries are almost trivial consequences of \eqref{e-herglotz} and \eqref{e-vartheta}.

Recall that for $f\in L^2(\mu_\gamma)$, the normalized Cauchy transform $\calC_{f \mu_\gamma}$ is analytic in the open unit disc $\D$. It turns out that, if we take $f(\xi) = \bar\xi$, then $\calC_{\bar\xi\mu_\gamma}$ is also uniformly bounded on $\D$.

\begin{cor}\label{c-key}
We have $\calC_{\bar\xi\mu_\gamma} \in H^\infty(\D)$.
\end{cor}

\begin{proof}
Since $|\theta(z)|\le 1$ is an analytic self-map of the disc, the origin is the only place where \eqref{e-vartheta} allows $\calC_{\bar\xi\mu_\gamma}$ to have a singularity. But at the origin we have $|\calC_{\bar\xi\mu_\gamma}(0)| = |\int \bar\xi d\mu|\le 1$, because $\mu$ is a probability measure.
\end{proof}

The following two implications of Theorem \ref{t-char1} yield well-known results in the theory of rank one perturbations.
For example, it is easy to verify Poltoratski's theorem in the case of $f = \bar\xi$.

\begin{cor}\label{c-polt}(Special case of Poltoratski's theorem, see Theorem \ref{t-polt})
For $\gamma\in\T$, the non-tangential limit satisfies
\[
 \lim_{z\to \zeta }\calC_{\bar\xi\mu_\gamma}(z) \to \bar \zeta  \qquad \text{for } \zeta \in\T\text{ a.e.}~(\mu_\gamma)\ti{s}.
\]
\end{cor}

\begin{proof}
 From \eqref{e-herglotz} and \eqref{e-vartheta} we obtain
\begin{align*}
 K\mu_\gamma(z) = (1- \bar\gamma\theta(z))^{-1}= (1- z\calC_{\bar\xi\mu_\gamma}(z))^{-1}.
\end{align*}
Clearly, we have $|K\mu_\gamma(z)| \to \infty$ as $z\to \zeta\in\T$ a.e.~$(\mu_\gamma)\ti{s}$. Hence, we have 
\[
|1- z\calC_{\bar\xi\mu_\gamma}(z)|\to 0
\]
as $z\to \zeta\in\T$ a.e.~$(\mu_\gamma)\ti{s}$.
\end{proof}

Since the characteristic function $\theta\in H^\infty(\D)$, we know that the non-tangential boundary values $\lim_{\D\ni z\to \zeta}\theta(z)$ exist $\zeta\in \T$ a.e.~$m$. The following result states that the non-tangential boundary limits exist on certain sets of Lebesgue measure zero, and what they are.

\begin{cor}
For every $\gamma\in \T$, the non-tangential boundary limits of the characteristic function $\theta(z)$ obey
	\[
	\lim_{z\to \zeta }\theta(z) = \gamma\qquad \text{for }\zeta \in\T \text{ a.e.}~(\mu_\gamma)\ti{s}.
	\] 
\end{cor}

\begin{proof}
By virtue of the special case of Poltoratski's theorem (see Corollary \ref{c-polt}), we have $\calC_{\bar\xi\mu_\gamma}(z)\to \bar \zeta$ as $z\to \zeta$ for $\zeta\in\T$ a.e.~$(\mu_\gamma)\ti{s}$, and it remains only to recall equality \eqref{e-vartheta}.
\end{proof}

\begin{rems}
(a) The latter corollary with an alternative proof can be found in \cite{cimaross}, see Corollary 9.1.24.\\
(b) Given a function $f\in H^\infty(\D)$ with $\|f\|\ci{H^\infty} \le 1$ and $f(0)=0$, one can use the latter corollary to show the existence (and obtain the value) of the non-tangential boundary value at some $\zeta\in\T$. Indeed, this can be done for some $|\gamma|=1$ by finding a corresponding Aleksandrov--Clark measure $\mu_\gamma$ that has a point mass at $\zeta$. One can then proceed in a similar way using the singular continuous part of $\mu_\gamma$.
\end{rems}

\subsection{Radial jump behavior of the normalized Cauchy transform across the unit circle}\label{s-jump}
It follows from Poltoratski's theorem that the non-tangential boundary values of the normalized Cauchy transform
\[
\lim_{z\to\xi, |z|<1}\calC_{f\mu}(z)
\qquad\text{and}
\qquad
\lim_{z\to\xi, |z|>1}\calC_{f\mu}(z)
\]
(from the inside and outside, respectively) coincide a.e.~$(\mu_\gamma)\ti{s}$. However, Poltoratski's theorem does not provide any information about the boundary values on the absolutely continuous part $(\mu_\gamma)\ti{ac}$. On that part of $\T$, the difference of the non-tangential boundary values of the normalized Cauchy transform (from the inside and outside, respectively) is some non-zero function, since $\calC_{f\mu}$ is an analytic function on $\C\backslash\T$. Equation \eqref{e-jump} below provides a simple explicit expression for this jump for non-tangential boundary limits.

The idea is to use the two representations
(equations \eqref{e-herglotz} and \eqref{e-vartheta}) of $\theta$
in order to obtain information about the jump behavior of the normalized Cauchy transform of the measure $\bar\xi \mu$.

For $\xi\in\T$ we use the notation
\[
 K^\pm\mu(\xi) = \lim_{z\to \xi, |z| \gtrless 1} K\mu(z)
\]
to denote the non-tangential boundary limits of the Cauchy transform from the outside and inside, respectively.
Similarly, for the non-tangential limits of the normalized Cauchy transform, consider
\[
 \calC_{f\mu}^\pm\mu(\xi) = \lim_{z\to\xi, |z| \gtrless 1} \calC_{f\mu}(z),
\]
where $f\in L^2(\mu)$. We denote the (non-tangential) jump of the normalized Cauchy transform across $\T$ at $\xi\in\T$ by
\begin{align*}
[[\calC_{f\mu}]] (\xi)= \calC^-_{f\mu}(\xi) - \calC^+_{f\mu}(\xi).
\end{align*}

\begin{theo}
For a non-negative Borel measure $\mu$, the non-tangential jump behavior of $\calC_{\bar\xi\mu}$ across the unit circle $\T$ is given by
\begin{align}\label{e-jump}
[[\calC_{\bar\xi\mu}]] (\xi) = \frac{1}{\xi(K^+\mu(\xi))(K^-\mu(\xi))}\, \frac{d\mu}{dm}(\xi)
\qquad\text{for }\xi\in\T\text{ a.e.~}m.
\end{align}
\end{theo}

\begin{proof}
Using the two representations of $\theta$, equations \eqref{e-herglotz} and \eqref{e-vartheta},
we see that
\begin{align}\label{overK}
   \calC_{\bar\xi\mu} (z) = \frac{1}{z}\left[1-\frac{1}{K\mu(z)}\right]
\qquad\text{for }z\in\C\backslash \T.
\end{align}
Recall that $\calC_{\bar\xi\mu}\in H^\infty(\D)$ by Corollary \ref{c-key}.

Re-grouping the limits (and noticing that $K^\pm\mu(\xi)\neq 0$ for $\xi\in\T$ a.e.~$m$), we obtain
\[
 [[\calC_{\bar\xi\mu}]] (\xi ) = \frac{1}{\xi }\left[\frac{1}{K^+\mu(\xi )}-\frac{1}{K^-\mu(\xi )}\right] = \frac{K^-\mu(\xi )\,-\,K^+\mu(\xi )}{\xi (K^+\mu(\xi ))(K^-\mu(\xi ))}
\]
for $\xi \in\T$ a.e.~$m$.

Recall Fatou's jump theorem (see for example, Corollary 2.4.2 of \cite{cimaross}) which is also known as Privalov's theorem, we have
\[
K^-\mu(\xi )\,-\,K^+\mu(\xi ) = \frac{d\mu}{dm}(\xi )\qquad\text{ for }\xi \in\T \text{ a.e.~}m.
\]

Application of this statement to the numerator yields the theorem.
\end{proof}

\begin{rems}
(a) Note that equality \eqref{overK} does not violate Poltoratski's theorem (Theorem \ref{t-polt}) because $|K\mu(\xi )|=\infty$ for $\xi \in\T$ a.e.~$\mu\ti{s}$.\\
(b) In order to verify that the latter theorem is reasonable in a special case, let us assume that the closed $\supp\mu\ti{ac}\neq \T$. Then it is well-known (see \cite{DSS}) that the characteristic function obeys $|\theta(\xi )|=1$ for $\xi \in \T\backslash \supp\mu\ti{ac}$ and that it possesses a so-called pseudo-analytic continuation ``across'' those points (in $\T\backslash \supp\mu\ti{ac}$). In particular, the normalized Cauchy transform cannot have a jump at those points. This is in agreement with the vanishing of the term $\frac{d\mu}{dm}(\xi )$ on the right hand side of equation \eqref{e-jump} at those points.
\end{rems}

\providecommand{\bysame}{\leavevmode\hbox to3em{\hrulefill}\thinspace}
\providecommand{\MR}{\relax\ifhmode\unskip\space\fi MR }
\providecommand{\MRhref}[2]{%
  \href{http://www.ams.org/mathscinet-getitem?mr=#1}{#2}
}
\providecommand{\href}[2]{#2}

\end{document}